\documentclass[12pt,a4paper]{amsart}
\usepackage{amsfonts}
\usepackage{amsthm}
\usepackage{amsmath}
\usepackage{amscd}
\usepackage{amssymb}
\usepackage[latin2]{inputenc}
\usepackage{t1enc}
\usepackage[mathscr]{eucal}
\usepackage{indentfirst}
\usepackage{graphicx}
\usepackage{graphics}
\usepackage{pict2e}
\usepackage{epic}
\usepackage{mathtools}
\numberwithin{equation}{section}
\usepackage[margin=2.9cm]{geometry}
\usepackage{epstopdf} 
\usepackage{tikz-cd}
\usepackage[all]{xy}

\makeatletter
\newtheorem*{rep@theorem}{\rep@title}
\newcommand{\newreptheorem}[2]{%
\newenvironment{rep#1}[1]{%
 \def\rep@title{#2 \ref{##1}}%
 \begin{rep@theorem}}%
 {\end{rep@theorem}}}
\makeatother

\newcommand{\Mod}[1]{\ (\mathrm{mod}\ #1)}

\newcommand{\simfin }{\sim_{\mathrm{fin}}}
\newcommand{\simd}{\sim_{\mathrm{d}}}

\theoremstyle{plain}
\newtheorem{Th}{Theorem}[section]
\newtheorem{Lemma}[Th]{Lemma}
\newtheorem{Cor}[Th]{Corollary}
\newtheorem{Prop}[Th]{Proposition}

 \theoremstyle{definition}
\newtheorem{Def}[Th]{Definition}

\newtheorem{Rem}[Th]{Remark}
\newtheorem{?}[Th]{Problem}

\newreptheorem{theorem}{Theorem}

\newreptheorem{lemma}{Lemma}
\newreptheorem{corollary}{Corollary}

\begin{document}

\title{Bounds for the rank of the finite part of operator $K$-Theory}

\author[S.K.Samurka\c{s}]{S\"{u}leyman Ka\u{g}an SAMURKA\c{S}}

\address{Texas A\&M University \\ Department of Mathematics \\
77840-College Station \\Texas \\USA} 

\email{kagan@math.tamu.edu}



\begin{abstract}
We derive a lower and an upper bound for the rank of the finite part of operator $K$-theory groups of maximal and reduced $C^*$-algebras of finitely generated groups.
The lower bound is based on the amount of polynomially growing conjugacy classes of finite order elements in the group.
The upper bound is based on the amount of torsion elements in the group.
We use the lower bound to give lower bounds for the structure group $S(M)$ and the group of positive scalar curvature metrics $P(M)$ for an oriented manifold $M$.

We define a class of groups called ``polynomially full groups'' for which the upper bound and the lower bound we derive are the same.
We show that the class of polynomially full groups contains all virtually nilpotent groups.
As example, we give explicit formulas for the ranks of the finite parts of operator $K$-theory groups for the finitely generated abelian groups, the symmetric groups and the dihedral groups.
\end{abstract}

\maketitle

\section{Introduction}
The purpose of this paper is to derive a lower bound for the rank of the finite part of the operator $K$-theory of maximal and reduced $C^{*}$-algebras, and use that lower bound to study the non-rigidity
of the manifolds and the space of positive scalar curvature metrics on a manifold.
Moreover, we derive an upper bound for the rank of the finite part, and introduce a class of groups called ``polynomially full groups'' for which the upper bound and the lower bound we derive are the same.

Given a manifold $M,$ we can ask the following question: How many ``distinct'' manifolds exist which are homotopy equivalent to $M?$
Two manifolds are considered to be ``distinct'' if they are not homeomorphic.
The answer to the question above is obviously related to the non-rigidity of the manifold $M.$
The more ``distinct'' manifolds homotopy equivalent to $M$ exists, the less ``rigid'' $M$ is.

Given a compact oriented manifold $M,$ the structure group $S(M)$ of $M$ is defined to be the abelian group generated by the equivalence classes of elements of the form $(f,M\textprime),$
where  $M\textprime$ is a compact oriented manifold and $f:M\textprime\rightarrow M$ is an orientation preserving homotopy equivalence.

If a compact smooth spin manifold $M$ has a positive scalar curvature metric and the dimension of $M$ is greater than or equal to $5,$
then $P(M)$ is defined to be the abelian group of equivalent classes of all positive scalar curvature metrics on M. For a more precise definition we refer to 
\cite[Section 4]{Yu}.

Weinberger, Xie and Yu \cite{WeinbergerXieYu} use the higher rho invariant to study the structure group $S(M).$

Weinberger and Yu \cite{Yu} use linearly independent elements in $K_{0}(C^{*}G)$ whose linear span intersects trivially with the image of the assembly map
$$\mu:K^{G}_{0}(EG)\rightarrow K_{0}(C^{*}G),$$
to give lower bounds for the ranks of the groups $S(M)$ and $P(M),$
where $G$ is the fundamental group of $M$ and $EG$ is the universal cover of the classifying space $BG.$

So in order to give lower bounds for the ranks of the groups $S(M)$ and $P(M),$ a general strategy would be to construct linearly independent elements 
in $K_{0}(C^{*}G)$ whose linear span intersects trivially with the image of the assembly map
$$\mu:K^{G}_{0}(EG)\rightarrow K_{0}(C^{*}G).$$

Given a finite order element $g\in G$ with $\operatorname{order}(g)=d,$ define 
\[p_g=\frac{1+g+g^2+\cdots +g^{d-1}}{d}\in\mathbb{C}G \subseteq C^{*}G.\]
It is not hard to show that $p_g$ is a projection (i.e. $p_g^2=p_g^*=p_g$).
So $p_g$ gives an element in $K_{0}(C^{*}G).$

For distinct finite order elements $g_1,g_2,\ldots ,g_n,$ in order to prove that $p_{g_1},p_{g_2},\ldots ,p_{g_n}$ are linearly independent in $K_{0}(C^{*}G),$
we can use homomorphisms
$$\rho_{i}:K_0(C^*G)\rightarrow \mathbb{C},$$
by mapping the $p_{g_j}$'s to $\mathbb{C}^n$ and showing that their images are linearly independent in $\mathbb{C}^n.$

So the problem reduces to find homomorphisms
$$\rho_{i}:K_0(C^*G)\rightarrow \mathbb{C}.$$
One way of finding such homomorphisms is to use trace maps on the algebra $C^*G.$
However, it is difficult (if possible) to construct different trace maps
$\tau:C^*G\rightarrow\mathbb{C}.$
Fortunately, we don't have to construct such traces.

We say a subalgebra $\mathcal{A}$ of $C^{*}_{r}G$ is smooth if it is stable under holomorphic functional calculus.
For a smooth dense subalgebra $\mathcal{A}$ of $C^{*}_{r}G,$ we have 
$$K_{0}(\mathcal{A})\cong K_{0}(C^{*}_{r}G),$$
where the isomorphism is induced by the inclusion map.
Hence, if we find trace maps 
$$\tau_{i}:\mathcal{A}\rightarrow\mathbb{C},$$
then they induce homomorphisms
$$\tau_{i}:K_{0}(C^{*}_{r}G)\cong K_{0}(\mathcal{A})\rightarrow\mathbb{C}.$$
Thus, composing with the homomorphism $K_{0}(C^{*}G)\rightarrow K_{0}(C^{*}_{r}G)$
we get homomorphisms
$$\rho_{i}:K_0(C^*G)\rightarrow \mathbb{C}.$$

For all $h\in G,$ let $\tau_{h} : \mathbb{C}G\rightarrow\mathbb{C}$ be defined as:
$$\tau_{h}\Big(\sum_{g\in G} a_{g}\cdot g\Big) := \sum_{g\in C(h)} a_{g},$$
where $C(h)$ is the conjugacy class of $h$ in $G.$
It is easy to see that $\tau_{h}$ is a trace map on $\mathbb{C}G$ \cite{St}.
So the problem reduces to lifting the $\tau_h$'s to trace maps on a suitable smooth and dense subalgebra $\mathcal{A}$ of $C^{*}_{r}G.$

For this we can define seminorms on $\mathbb{C}G$ and take the completion of $\mathbb{C}G$ with respect to those seminorms.
The completion should be as small as possible so that we can lift the $\tau_h$'s.

Now we describe our results in more detail.

For a group $G,$ let $G^{fin}$ be the subset of $G$ of finite order elements. Now define a relation $\simfin $ on $G^{fin}$ as follows:
$g\simfin  h$ 
if and only if there exists $\gamma\in G$ such that $p_g=\gamma p_h\gamma^{-1}$. It is easy to see that $\simfin$ is an equivalence relation on $G^{fin}.$ Now we define 
$$\mathcal{F}_G:=|G^{fin}/_{\simfin} |.$$
In the following, we give an equivalent definition for $\mathcal{F}_G$ that we use in the proofs.
For all $d \in \mathbb{N}$ define $G^{fin}_{d}:=\{g \in G^{fin}\ |\ \operatorname{order}(g)=d \}.$

In the following, we define an equivalence relation on $G^{fin}_{d}.$
\begin{Def}
For a group G we define the relation $\simd $ on $G^{fin}_{d}$ as follows: 
$g\simd h$ if and only if $\exists\ a\in\mathbb{N}$ such that $g^{a} \in C(h)$, where $C(h)=\{f h f^{-1} | f \in G\}.$
\end{Def}

It is easy to verify that $\simd $ is an equivalence relation on $G^{fin}_{d}.$

Next, we give an equivalent definition for $\mathcal{F}_G.$

\begin{Def}
Let $\widetilde{G}^{fin}_{d} := G^{fin}_{d}/_{\simd} .$ 
Define $\mathcal{F}_G:=\sum_{d=1}^{\infty} |\widetilde{G}^{fin}_{d}|.$
\end{Def}

\begin{Rem}
It is not hard to see that given $g,h\in G^{fin}$ we have 
$$g\simfin h \Leftrightarrow \exists d\in\mathbb{N},\ \operatorname{order}(g)=\operatorname{order}(h)=d \text{ and } g\simd h.$$
So the two definitions of $\mathcal{F}_G$ are the same.
\end{Rem}

Following Weinberger and Yu \cite{Yu}, we define $K_0^{fin}(C^* G)$ to be the subgroup of $K_0(C^* G)$ generated by the set 
$$\{[p_g]\ :\ g\in G^{fin}\},$$
where $[p_g]$ denotes the class of the projection $p_g$ in $K_0(C^* G).$ We define $K_0^{fin}(C^{*}_r G)$ similarly. 
Given $g,h\in G^{fin}$ with $g\simfin h,$ we have $[p_g]=[p_h]\in K_0(C^* G).$ Hence, $K_0^{fin}(C^* G)$ has rank at most $\mathcal{F}_G.$
Using the natural surjection 
$$K_0^{fin}(C^* G)\twoheadrightarrow K_0^{fin}(C_r^* G)$$
induced by the identity map $\operatorname{id}:\mathbb{C}G\rightarrow\mathbb{C}G,$ we can conclude that $K_0^{fin}(C_r^* G)$ also has rank at most $\mathcal{F}_G.$

The following result is proved in Section ~\ref{frameworkproof}.

\begin{Th}\label{framework}
 Let $S\subseteq G^{fin}.$ If there exists a smooth subalgebra $\mathcal{A}$ of $C^{*}_{r}G$ containing $\mathbb{C}G$ 
 and if $\forall h\in S$ there exists a trace function 
 $$\widetilde{\tau}_{h}:\mathcal{A}\rightarrow\mathbb{C}$$
 extending the trace function $\tau_{h}:\mathbb{C}G\rightarrow\mathbb{C},$
 then we have
 $$\operatorname{rank}(K^{fin}_0(C^{*}_{r}G))\geq |S/_{\simfin} |$$
 and for the assembly map $\mu:K^{G}_{0}(EG)\rightarrow K_{0}(C^{*}G),$
 we have 
 $$\operatorname{\operatorname{Im}}\mu\cap K^{S}_{0}(C^{*}G)=\{0\},$$
 where $K^{S}_{0}(C^{*}G)$ is the subgroup of $K_{0}(C^{*}G)$ generated by the set $\{[p_{g}]:g\in S\}$ and 
 $EG$ is the universal cover of the classifying space $BG.$
\end{Th}

Given a finitely generated group $G$ with word length norm $\lVert . \rVert_w$ and given $h\in G,$ we define $C(h):=\{ghg^{-1}\ :\ g\in G\},$  
$C_l(h):=\{g\in C(h):\lVert g\rVert_w=l\}$ and $n_{h,l}:=|C_l(h)|.$ 
\begin{Def}
We say that $C(h)$ has polynomial growth, if $\exists c\in\mathbb R_{>0}$ $\exists d\in\mathbb{N}$ such that $n_{h,l}\leq c\cdot l^d$ for all $l\in\mathbb{N}$ 
and define $G^{pol}:=\{g\in G\ |\ C(g) \text{ has polynomial growth}\}.$
\end{Def}

In the following, we define the maximum number of non-equivalent finite order elements we can choose from $G^{pol}.$
\begin{Def}
Let $G$ be a finitely generated group. We define 
$$\mathcal{F}^{pol}_G := |(G^{pol}\cap G^{fin})/_{\simfin} |.$$
\end{Def}

In Section ~\ref{polynomialproof}, we prove that the hypotheses of Theorem ~\ref{framework} are satisfied for $S=G^{pol}\cap G^{fin}$ and $\mathcal{A}=C^{pol}_{S}G$ defined in the proof.

Hence, we get the following main result about polynomial growth in conjugacy classes:

\begin{Th} \label{mainpol}
Let $G$ be a finitely generated group. We have 
$$\mathcal{F}^{pol}_G\leq\operatorname{rank}(K^{fin}_0(C^{*}_r G))\leq\operatorname{rank}(K^{fin}_0(C^{*} G))\leq\mathcal{F}_G$$
and for the assembly map $\mu : K_{0}^{G}(EG)\rightarrow K_{0}(C^{*}G),$ we have 
$$\operatorname{\operatorname{Im}}(\mu )\cap K_{0}^{fin,pol}(C^{*}G)=\{0\},$$
where $K_{0}^{fin,pol}(C^{*}G)$ is the subgroup of $K_{0}(C^{*}G)$ generated by the set 
$$\{[p_g]\ :\ g\in G^{fin}\cap G^{pol}\}.$$
\end{Th}

The proof of this result is in Section ~\ref{polynomialproof} of this paper. Note that this result follows from the injectivity part of the Baum-Connes conjecture ~\cite{BaumConnes}.

Gong \cite{Gong} finds a lower bound for the rank of $K_0^{fin}(C_r^* G)$ for the groups with property (RD) and conjugacy classes having poynomial growth.
In our results, we don't require property (RD) and also improve the lower bound.

The importance of these results lies in the following:

\begin{reptheorem}{structuregroup}\cite{Yu}
Let $M$ be a compact oriented manifold with dimension $4k-1$ $(k>1).$ Suppose $\pi_1(M)=G$ and $g_1,\cdots, g_n$ be finite order elements in $G$ such that $g_i\neq e$
for all $i$ and $\{[p_{g_1}],\cdots,[p_{g_n}]\}$ generates an abelian group of $K_0(C^{*}G)$ with rank n. Suppose that any nonzero element in the abelian subgroup of 
$K_0(C^{*}G)$ generated by $\{[p_{g_1}],\cdots,[p_{g_n}]\}$ is not in the image of the map $\mu : K_{0}^{G}(EG)\rightarrow K_{0}(C^{*}G)$, then the rank of the structure group $S(M)$ is greater than or equal to n.
\end{reptheorem}

Combining this result with our results we get:

\begin{repcorollary}{structuregroupour}
For a compact oriented manifold $M$ with dimension $4k-1$ $(k>1),$
the rank of the structure group $S(M)$ is greater than or equal to $\mathcal{F}^{pol}_G-1,$
where $G=\pi_1(M).$
\end{repcorollary}

Another application of our results is the following:

Let $r_{fin}(G)$ be the rank of the abelian group $K^{fin}_{0}(C^{*}G)$ generated
by $[p_g]$ for all finite order elements $g\in G.$ Here $g$ is allowed to be the
identity element $e.$ So we have $\mathcal{F}_G\geq r_{fin}(G)= \operatorname{rank}(K^{fin}_{0}(C^{*}G))\geq\mathcal{F}^{pol}_G.$

\begin{reptheorem}{positivescalar}\cite{Yu}
  \begin{enumerate}
    \item
    Let $M$ be a compact smooth spin manifold with a positive scalar curvature metric and dimension $2k-1$ $(k>2).$
    The rank of the abelian group $P(M)$ is greater than or equal to $r_{fin}(G)-1.$

    \item
    Let $M$ be a compact smooth spin manifold with a positive scalar curvature metric and dimension $4k-1$ $(k>1).$
    The rank of the abelian group $P(M)$ is greater than or equal to $r_{fin}(G).$
  \end{enumerate}
\end{reptheorem}

Combining this result with our results, we get:

\begin{repcorollary}{positivescalarour}
Let $M$ be a compact smooth spin manifold with a positive scalar curvature metric.
Let $G=\pi_1(M).$
  \begin{enumerate}
    \item
    If $M$ has dimension $2k-1$ $(k>2),$ then the rank of the abelian group $P(M)$ is greater than or equal to $\mathcal{F}^{pol}_G-1.$

    \item
    If $M$ has dimension $4k-1$ $(k>1),$ then the rank of the abelian group $P(M)$ is greater than or equal to $\mathcal{F}^{pol}_G.$
  \end{enumerate}

\end{repcorollary}

This paper consists of 5 sections (including the introduction):
\begin{itemize}
  \item In Section 2, we prove our framework theorem (Theorem ~\ref{framework}).

  \item In Section 3, we recall dominating functions from \cite{Roe}. As Engel did in \cite{Engel}, using the dominating functions, we define a seminorm $\lVert .\rVert_{\mu,h}$ on $\mathbb{C}G$ for each $h\in G^{pol}.$
  Using the seminorm and the operator norm, we complete $\mathbb{C}G$ and get a smooth dense subalgebra of $C^*_r G.$ We call that algebra $C^{pol}_hG.$ 
  We also recall the trace functions on $\mathbb{C}G$ corresponding to an element in $G.$ Using the properties of the seminorms, we lift $\tau_h$ to a trace function $\widetilde{\tau}_h$ on $C^{pol}_{h}G.$

  \item In Section 4, we prove Theorem ~\ref{mainpol}. As applications, we derive lower bounds for the ranks of the structure group and the group of positive scalar curvature metrics of manifolds.

  \item In Section 5, we define the class of polynomially full groups.
  We show that subgroups, products, and finite extensions of polynomially full groups are also polynomially full.
  For polynomially full group $G,$ we show that
  $$K^{fin}_{0}(C^{*}_{r}G)\cong K^{fin}_{0}(C^{*}G)\cong\bigoplus_{i=1}^{\mathcal{F}_G}\mathbb{Z}.$$
  The class of polynomially full groups includes trivially all finite groups and finitely generated torsion-free groups. We show that it also includes all finitely generated virtually nilpotent groups.
  At the end of the section, we derive formulas for the number $\mathcal{F}_G,$ where $G$ is finitely generated abelian group, dihedral group, or symmetric group.
\end{itemize}

\section*{Acknowledgements}
The author would like to acknowledge Guoliang Yu for his invaluable guidance.
The author would also like to acknowledge Alexander Engel and Bogdan Nica for their valuable suggestions.

\section{Proof of Theorem ~\ref{framework}}\label{frameworkproof}

\begin{proof}
Since $\mathcal{A}$ is a smooth and dense subalgebra of $C^{*}_{r}G,$ we have 
$$K_{0}(\mathcal{A}) \cong K_{0}(C^{*}_{r}G),$$
where the isomorphism is induced by the inclusion map
$$i : \mathcal{A} \rightarrow C^{*}_{r}G\ .$$
Since the finite parts of $K_{0}(\mathcal{A})$ and $K_{0}(C^{*}_{r}G)$ are coming from $\mathbb{C}G,$
we have $$K_{0}^{fin}(\mathcal{A})\cong K_{0}^{fin}(C^{*}_{r}G).$$

Hence, for the first part of the theorem, it suffices to show that 
$$\operatorname{rank}(K_{0}^{fin}(\mathcal{A}))\geq |S/_{\simfin} |.$$
Let $K_{0}^{S}(\mathcal{A})$ be the subgroup of $K_{0}^{fin}(\mathcal{A})$ generated by the set $\{[p_g]:g\in S\}.$
Thus, it suffices to show that $\operatorname{rank}(K_{0}^{S}(\mathcal{A}))\geq |S/_{\simfin} |.$

Let $\{s_1,...,s_n\}$ be an arbitrary subset of $S$ such that, we have $s_i\nsim_{fin} s_j$ for $i\neq j.$
We are going to show that, the subgroup of $K_{0}^{S}(\mathcal{A})$ generated by the set $\{[p_{s_1}],\cdots [p_{s_n}]\}$
has rank $n.$ Therefore, we are going to conclude that 
$$\operatorname{rank}(K_{0}^{S}(\mathcal{A}))\geq |S/_{\simfin} |.$$

For all $i\in\{1,2,\cdots,n\}\ let\ d_i=\operatorname{order}(s_i)$ and assume $d_1\leq d_2\leq\cdots\leq d_n.$ 
We have all the traces $\widetilde{\tau}_{s_{i}} : \mathcal{A} \rightarrow \mathbb{C}$ defined. This gives us the homomorphisms (with abuse of notation)
$$\widetilde{\tau}_{s_{i}} : K_{0}^{S}(\mathcal{A}) \rightarrow \mathbb{C}.$$
Now define $$\mathcal{M}_{n}=
	\begin{pmatrix}
		\widetilde{\tau}_{s_{1}}([p_{s_{1}}]) & \widetilde{\tau}_{s_{1}}([p_{s_{2}}]) & \cdots & \widetilde{\tau}_{s_{1}}([p_{s_{n}}]) \\
		\widetilde{\tau}_{s_{2}}([p_{s_{1}}]) & \widetilde{\tau}_{s_{2}}([p_{s_{2}}]) & \cdots & \widetilde{\tau}_{s_{2}}([p_{s_{n}}]) \\
		\vdots  & \vdots  & \ddots & \vdots  \\
		\widetilde{\tau}_{s_{n}}([p_{s_{1}}]) & \widetilde{\tau}_{s_{n}}([p_{s_{2}}]) & \cdots & \widetilde{\tau}_{s_{n}}([p_{s_{n}}]) 
	\end{pmatrix},$$
	
where $p_{s_{j}}=\frac{1+s_{j}+...+s_{j}^{d_{j}-1}}{d_j}\in\mathbb{C}G\subseteq \mathcal{A}$ and $[p_{s_{j}}]$ shows the class in $K_{0}^{S}(\mathcal{A})$ represented by the projection $p_{s_{j}}.$
Now $\forall i,j\in\{1,2,...,n\}$ with $i>j$, there are 2 cases:
\subsection*{Case 1 ($d_i>d_j$)}
In this case, we have
$$\widetilde{\tau}_{s_{i}}([p_{s_{j}}])=\tau_{s_{i}}(p_{s_{j}})=\tau_{s_{i}}\Bigg(\frac{1+s_{j}+...+s_{j}^{d_{j}-1}}{d_j}\Bigg)$$
and since $\forall a\in\mathbb{N}\ $
$\operatorname{order}(s_{j}^{a})\leq \operatorname{order}(s_{j})=d_j<d_i=\operatorname{order}(s_{i}),$
we have\\
$\forall a\in\mathbb{N}\ s_{j}^{a}\notin C(s_i)$ (all elements from $C(s_i)$ have order $d_i$).
Thus, $\widetilde{\tau}_{s_{i}}([p_{s_{j}}])=0.$

\subsection*{Case 2 $(d_{i}=d_{j})$ and $s_{j} \nsim_{d} s_{i}$}
In this case, we have $\forall a\in\mathbb{N}\ s_{j}^{a}\notin C(s_{i})$ by definition of $\simd .$
So $\widetilde{\tau}_{s_{i}}([p_{s_{j}}])=0.$

Hence, $\mathcal{M}_{n}$ is an upper triangular matrix.
\\
Now $\forall i\in \{1,2,...,n\},$ we have $s_i\in C(s_i)$ so, 

\begin{align*}
	\widetilde{\tau}_{s_{i}}([p_{s_{i}}])	&=\tau_{s_{i}}(p_{s_{i}})\\
						&=\tau_{s_{i}}\Bigg(\frac{1+s_{i}+...+s_{i}^{d_{i}-1}}{d_i}\Bigg)\\
						&\geq \frac{1}{d_i}\ .
\end{align*}

Thus, the elements in the diagonal of $\mathcal{M}_{n}$ are non-zero. Hence, $det(\mathcal{M}_{n})\neq 0.$
So $\mathcal{M}_{n}$ has full rank. Thus, in $K_{0}^{S}(\mathcal{A})$, the elements 
$[p_{s_1}],...,[p_{s_n}]$ are $\mathbb{Z}$-linearly independent. Therefore, $\operatorname{rank}(K_{0}^{S}(\mathcal{A})) = n.$
Thus, we get $$\operatorname{rank}(K_{0}^{fin}(\mathcal{A}))\geq |S/_{\simfin} |.$$

Now, let's make some preliminary definitions for the proof of the second part of the Theorem ~\ref{framework}:

Let $\mathcal{H}$ be an infinite dimensional separable Hilbert space. Let $\mathcal{B}(\mathcal{H})$ be the set of bounded linear
operators on $\mathcal{H}.$
We define $\mathcal{S}_p\ :=\ \{T\in\mathcal{B}(\mathcal{H})\ |\ tr((T^{*}T)^{\frac{p}{2}})<\infty\},$
where $tr(P)\ :=\ \sum_{n\in\mathbb{N}}\langle Pe_n, e_n\rangle$ for an orthonormal basis $\{e_n\}_{n=1}^{\infty}$ and for a bounded linear operator $P\in\mathcal{B}(\mathcal{H}).$
We remark that, trace does not depend on the particular choice of an orthonormal basis. 
We call $\mathcal{S}_p$ the ring of Schatten $p$-class operators on an infinite dimensional and separable Hilbert space.
Now define $\mathcal{S}\ :=\ \cup_{p=1}^{\infty} \mathcal{S}_p.$
The ring $\mathcal{S}$ is called the ring of Schatten class operators. Let $\mathcal{S}G$ be the group algebra over the ring $\mathcal{S}$ \cite{Yu2}.
Let $j:\mathbb{C}G\rightarrow\mathcal{S}G$ be the inclusion homomorphism defined by:
$$j(a)=p_0a$$
for all $a\in\mathbb{C}G,$ where $p_0$ is a rank one projection in $\mathcal{S}.$

In the following, we show that nonzero elements in the finite part of $K_0(C^{*}G)$ generated by the set $\{[p_g]:g\in S\}$ are not in the image of the assembly map
$\mu : K_{0}^{G}(EG)\rightarrow K_{0}(C^{*}G),$ where $EG$ is the universal space for proper and free $G$-action.
In the proof, we use the $n$-cocycle $\tau_g^{(n)}$ on $\mathcal{S}_m G$ introduced in \cite{Yu}, which gives in some sense the extension of the classical trace $\tau_g.$
So we have a commutative diagram
$$\xymatrix{
  &K_0(\mathcal{S}_m G) \ar[r]^{\psi} \ar[d]_{(\tau_g^{(n)})_{*}}	&K_0(C^{*}G) \ar[dl]^{\widetilde{\tau}_g}\\
  & \mathbb{C}
}$$
where (with abuse of notation) $\widetilde{\tau}_g: K_0(C^{*} G)\rightarrow\mathbb{C}$ is the pullback of the homomorphism $\widetilde{\tau}_g: K_0(C^{*}_{r} G)\rightarrow\mathbb{C}.$
Recall that $K_{0}(\mathcal{A}) \cong K_{0}(C^{*}_{r}G).$

Assume there exists a non-zero $z\in \operatorname{Im}(\mu )\cap K_{0}^{S}(C^{*}G).$ Then $z=\sum_{i=1}^{s} c_i\cdot [p_{g_i}]$ for some pairwise non-equivalent $g_1,\cdots,g_s \in S$ and 
$c_1,\cdots,c_s\in\mathbb{Z}\setminus\{0\}.$ For all $i\in\{1,\cdots,s\}$ let $d_i=\operatorname{order}(g_i).$ Without loss of generality, we can assume $d_1\leq\cdots\leq d_s.$
Now let $g=g_s.$

Now let $z'=\sum_{i=1}^{s} c_i\cdot [j(p_{g_i})].$ We have $z'\in K_0(\mathcal{S}_m G)$ for some $m\in\mathbb{N}.$
Then we get $z'\in \operatorname{Im}(A),$ where 
$$A:H_0^{Or G}(EG, \mathbb{K}(\mathcal{S}_m)^{-\infty})\rightarrow K_0(\mathcal{S}_m G)$$
is the assembly map.

Let $n=2k$ be the smallest even number greater than or equal to $m.$
Define an $n$-cocycle $\tau_g^{(n)}$ on $\mathcal{S}_m G$ by:
$$\tau_g^{(n)}(a_0,a_1,\cdots ,a_n):=\sum_{\gamma\in C(g)}tr(\gamma^{-1}a_0a_1\cdots a_n)$$
for all $a_i\in\mathcal{S}_m G,$ where $tr:\mathcal{S}_1 G\rightarrow\mathbb{C},$ is the trace defined by:
$$tr(\sum_{\gamma\in G}b_{\gamma}\gamma):=trace(b_e)\ .$$
Since $\tau_g^{(n)}$ is an $n$-cocycle, it induces a homomorphism
$$(\tau_g^{(n)})_{*}:K_0(\mathcal{S}_m G)\rightarrow\mathbb{C}.$$

It is shown in \cite{Yu} that $(\tau_g^{(n)})_{*}([j(p)])= \tau_g(p)$ for all projections $p\in\mathbb{C}G$ and $(\tau_g^{(n)})_{*}(z')=0.$
So we have the commutative diagram
$$\xymatrix{
  &K_0(\mathcal{S}_m G) \ar[r]^{\psi} \ar[d]_{(\tau_g^{(n)})_{*}}	&K_0(C^{*}G) \ar[dl]^{\widetilde{\tau}_g}\\
  & \mathbb{C}
}$$
where (with abuse of notation) $\widetilde{\tau}_g: K_0(C^{*} G)\rightarrow\mathbb{C}$ is the pullback of the homomorphism $\widetilde{\tau}_g: K_0(C^{*}_{r} G)\rightarrow\mathbb{C}.$ We have $\psi(z')=z.$ Hence, we get 
$$\widetilde{\tau}_g(z)=\widetilde{\tau}_g(\psi(z'))=(\tau_g^{(n)})_{*}(z')=0.$$
However, we have $\widetilde{\tau}_g(z)=\widetilde{\tau}_g(\sum_{i=1}^{s} c_i\cdot [p_{g_i}])=\sum_{i=1}^{s} c_i\cdot \widetilde{\tau}_g([p_{g_i}])=k\cdot\frac{c_s}{d_s}$ for some $k\in\mathbb{N}.$
Thus, $\widetilde{\tau}_g(z)\neq 0.$
Contradiction shows that $\operatorname{Im}(\mu )\cap K_{0}^{S}(C^{*}G)=\{0\}.$
\end{proof}

\begin{Rem}
For any group $G,$ we can build up a matrix similar to the matrix in the proof of Theorem ~\ref{framework}, and show that
$p_{g}'s$ corresponding to pairwise non-equivalent $g's$ in $G^{fin}$ are linearly independent in $K_0(\mathbb{C}G),$ since all the traces 
$$\tau_g:\mathbb{C}G\rightarrow\mathbb{C}$$
are already defined. Hence, we can conclude that
$$K_0^{fin}(\mathbb{C}G)\cong\bigoplus_{i=1}^{\mathcal{F}_G}\mathbb{Z},$$
as soon as $\mathcal{F}_G\leq\aleph_0,$
where $\aleph_0$ is the cardinality of the set of the natural numbers $\mathbb{N},$
and $K_0^{fin}(\mathbb{C}G)$ is the subgroup of $K_0(\mathbb{C}G)$ generated by the idempotents
$$\{[p_g]:g\in G^{fin}\}.$$
\end{Rem}

\section{Dominating Functions, Seminorms, and Trace Functions}
In the first part of this section, we recall dominating functions from \cite{Roe}. As Engel did in \cite{Engel}, using the dominating functions, we define a seminorm $\lVert .\rVert_{\mu,h}$ on $\mathbb{C}G$ for each $h\in G^{pol}.$
Using the seminorms and the operator norm, we complete $\mathbb{C}G$ and get a smooth dense subalgebra of $C^*_r G.$ We call that algebra $C^{pol}_hG.$ 
In the second part of this section, we recall the trace functions on $\mathbb{C}G$ corresponding to an element in $G.$ Using the properties of the seminorms, we lift $\tau_h$ to a trace function $\widetilde{\tau}_h$ on $C^{pol}_{h}G$ for each $h\in G^{pol}.$

\subsection{Dominating Functions}
In the first part of this section, we recall the dominating functions, prove some properties about them, and using those functions, we define seminorms on $\mathbb{C}G.$
Completing $\mathbb{C}G$ with respect to those seminorms and the operator norm, we construct smooth dense subalgebras $C^{pol}_{h}G$ of $C^{*}_{r} G$ for $h\in G^{pol}.$

In the following, we recall preliminary notions for the definition of the dominating functions.
We use $\mathbb R_{>0}$ and $\mathbb R_{\geq 0}$ for the sets of positive and non-negative real numbers, respectively.

\begin{Def}
Given $u\in\ell^2 G$ define $\operatorname{Supp}\ u:=\{g\in G\ :\ u(g)\neq 0\}.$
Now for all $S\subseteq G$ and $R\in\mathbb R_{\geq 0}$, define $B_R(S):=\{g\in G\ :\ d_w(g,S)\leq R\},$
where $d_w$ is the metric induced by $\lVert .\rVert_{w}.$
Define $\lVert u\rVert_{S}:=(\sum_{g\in S}\lvert u(g)\rvert^2)^\frac{1}{2}.$
\end{Def}

In the following, we recall the dominating function $\mu_{A}$ for an operator $A\in\mathcal{B}(\ell^{2}G).$ We use these dominating functions to define some seminorms on $\mathbb{C} G.$
\begin{Def} \cite{Engel}
For all $A\in\mathcal{B}(\ell^{2}G)$ define $\mu_{A}:\mathbb R_{>0}\rightarrow\mathbb R_{\geq 0}$ as
$$\mu_{A}(R):=\inf\{C\in\mathbb R_{>0}\ :\ \lVert Au\rVert_{G\setminus B_{R}(\operatorname{Supp}\ u)}\leq C\cdot \lVert u\rVert\text{, for all }u\in\ell^2 G\}.$$
\end{Def}

The following is a triangular inequality result we use at several places in our paper.

\begin{Lemma}
For all $A,B\in\mathcal{B}(\ell^2 G)$ and $R\in\mathbb R_{>0}$ we have $\mu_{A+B}(R)\leq\mu_{A}(R)+\mu_{B}(R).$
\end{Lemma}
\begin{proof}
For all $R\in\mathbb R_{>0}\text{ and }u\in\ell^2 G,$ we have 
\begin{align*}
\lVert (A+B)u\rVert_{G\setminus B_R(Supp\ u)}	&\leq \lVert Au\rVert_{G\setminus B_R(Supp\ u)}+\lVert Bu\rVert_{G\setminus B_R(Supp\ u)}\\
						&\leq\mu_A(R)\cdot \lVert u\rVert + \mu_B(R)\cdot \lVert u\rVert\\
						&=(\mu_A(R)+\mu_B(R))\cdot \lVert u\rVert\ .
\end{align*}
Thus, we get $\mu_{A+B}(R)\leq\mu_{A}(R)+\mu_{B}(R)$ for all $R\in\mathbb R_{>0}.$
\end{proof}

In the following, we estimate the dominating function with the operator norm.

\begin{Lemma}
For all $A\in\mathcal{B}(\ell^2 G)$, we have $\mu_{A}(R)\leq\lVert A\rVert_{op}$ for all $R\in\mathbb R_{>0}.$
\end{Lemma}
\begin{proof}
For all $A\in\mathcal{B}(\ell^2 G),\ R\in\mathbb R_{>0},\ u\in\ell^2 G,$ we have
$$\lVert Au\rVert_{G\setminus B_R(Supp\ u)}\leq \lVert Au\rVert\leq\lVert A\rVert_{op}\cdot \lVert u\rVert.$$
Thus, we get $\mu_A(R)\leq\lVert A\rVert_{op}.$
\end{proof}

In the following, we use the previous estimate to show that convergence in the operator norm implies point-wise convergence in the dominating functions.
We use this result in the proof of the smoothness of the subalgebras $C^{pol}_{h} G$ of $C^{*}_{r}G$ for $h\in G^{pol}.$

\begin{Lemma}\label{point}
Let $\{A_{n}\}_{n=1}^{\infty}$ be a sequence of operators in $\mathcal{B}(\ell^2 G)$ converging (in $\lVert.\rVert_{op}$ norm) to $A\in\mathcal{B}(\ell^2 G).$
Then $\{\mu_{A_{n}}\}_{n=1}^{\infty}$ converges to $\mu_{A}$ point-wise.
\end{Lemma}

\begin{proof}
Given $R \in\mathbb R_{>0}$, we have
\begin{align*}
\mu_{A_n}(R)	&=\mu_{A+(A_n-A)}(R)\\
							&\leq\mu_{A}(R)+\mu_{A_n-A}(R)\\
							&\leq\mu_{A}(R)+\lVert A_n-A\rVert_{op}\ .
\end{align*}
Similarly, we get $\mu_{A}(R)\leq\mu_{A_n}(R)+\lVert A-A_n\rVert_{op}.$
Thus, we have
$$\mu_{A}(R)-\lVert A-A_n\rVert_{op}\leq\mu_{A_n}(R)\leq\mu_{A}(R)+\lVert A_n-A\rVert_{op}.$$
Hence, we get $\lim_{n\to\infty} \mu_{A_n}(R)=\mu_{A}(R),\ \forall R\in\mathbb R_{>0}.$
\end{proof}

\begin{Rem}
Actually we have uniform convergence of $\{\mu_{A_{n}}\}_{n=1}^{\infty}$ to $\mu_{A}.$ However, point-wise convergence is enough for our purposes.
\end{Rem}

In the following, we are defining the seminorms we use to build the smooth dense subalgebras $C^{pol}_{h} G$ of $C^{*}_{r}G$ for $h\in G^{pol}.$

\begin{Def}
Recall that given $h\in G^{pol}$ $\exists C_{h}\in\mathbb R_{> 0}$ and $d_{h}\in\mathbb{N}$ such that $\forall l\in\mathbb{N}$ we have 
$n_{h,l}\leq C_{h}\cdot l^{d_{h}}.$ Let $b_{h}$ be a natural number greater than or equal to $\frac{d_{h}}{2}+2.$ 
Now define \footnote{We use a notation different than in \cite{Engel}.}
$$\lVert A \rVert_{\mu,h}:=\inf\{D\in\mathbb R_{> 0}\ |\ \mu_{A}(R)\leq D\cdot R^{-b_{h}}\ \forall R\in\mathbb R_{> 0}\}.$$
\end{Def}

\begin{Lemma}
$\lVert . \rVert_{\mu,h}$ is a seminorm on $\mathbb{C}G.$
\end{Lemma}

\begin{proof}
For all $A\in\mathbb{C}G,$ we have obviously $\lVert A\rVert_{\mu,h}\geq 0.$

Now for all $A,B\in\mathbb{C}G$ and $R>0$ we have
\begin{align*}
	\mu_{A+B}(R)	&\leq\mu_{A}(R)+\mu_{B}(R) \\
			&\leq\lVert A\rVert_{\mu,h}R^{-b_{h}}+\lVert B\rVert_{\mu,h}R^{-b_{h}}\\
			&=(\lVert A\rVert_{\mu,h}+\lVert B\rVert_{\mu,h})R^{-b_{h}}.
\end{align*}
Therefore $\lVert A+B\rVert_{\mu,h}\leq \lVert A\rVert_{\mu,h}+\lVert B\rVert_{\mu,h}.$
Hence, $\lVert . \rVert_{\mu,h}$ is a seminorm on $\mathbb{C}G.$
\end{proof}

In the following, we define our main gadget (a smooth dense subalgebra of $C^{*}_{r}G$).
We use the properties of the seminorm to lift the trace function $\tau_h$ (originally on $\mathbb{C}G$) to this subalgebra of $C^{*}_{r}G.$

\begin{Def}
For each $h\in G^{pol},$ we define $C^{pol}_{h}G$ as the completion of $\mathbb{C}G$ with respect to the norm $\lVert . \rVert_{op}$ and the seminorm $\lVert . \rVert_{\mu,h}.$
\end{Def}

Since $C^{pol}_{h}G$ contains $\mathbb{C}G$, it is dense (in the operator norm) in $C^{*}_{r}G.$

In the following, we show that $C^{pol}_{h}G$ is an algebra over the complex numbers. The only nontrivial part is to show that it is closed under multiplication.

\begin{Lemma}
$C^{pol}_{h}G$ is an algebra over $\mathbb{C}$.
\end{Lemma}

\begin{proof}
Given $A,B \in C^{pol}_{h}G,$ there exist sequences $\{A_n\}_{n=1}^{\infty}$ and $\{B_n\}_{n=1}^{\infty}$ in $\mathbb{C}G$
converging (in both norms) to $A$ and $B$ respectively.

The only nontrivial part is to show that $\lim_{n\to\infty}\lVert AB-A_nB_n\rVert_{\mu,h}=0.$
We have
\begin{align*}
\lVert AB-A_nB_n\rVert_{\mu,h}	&=\lVert (AB-A_nB)+(A_nB-A_nB_n)\rVert_{\mu,h}\\
				&=\lVert (A-A_n)B+A_n(B-B_n)\rVert_{\mu,h}\\
				&\leq\lVert (A-A_n)B\rVert_{\mu,h}+\lVert A_n(B-B_n)\rVert_{\mu,h}\ .
\end{align*}

We only show $\lim_{n\to\infty}\lVert (A-A_n)B\rVert_{\mu,h}=0:$

Let $C_n=A-A_n$ then, for all $R\in\mathbb R_{> 0},$ we have (the first inequality below is from \cite[Prop. 5.2]{Roe})
\begin{align*} 
	\mu_{C_nB}(R) 	&\leq 2\lVert C_n \rVert_{op}\mu_{B}(R/2)+\lVert B \rVert_{op}\mu_{C_n}(R/2)+2\mu_{C_n}(R/2)\mu_{B}(R/2)\\
			&\leq 2\lVert C_n \rVert_{op}\lVert B \rVert_{\mu,h}(R/2)^{-b_{h}}+\lVert B \rVert_{op}\lVert C_n \rVert_{\mu,h}(R/2)^{-b_{h}}+
			2\lVert C_n\rVert_{\mu,h}\lVert B \rVert_{\mu,h}(R/2)^{-2b_{h}}\\
			&=\{2^{b_{h}+1}\lVert C_n \rVert_{op}\lVert B \rVert_{\mu,h}+2^{b_{h}}\lVert C_n \rVert_{\mu,h}\lVert B \rVert_{op}+2^{2b_{h}+1}\lVert C_n \rVert_{\mu,h}\lVert B \rVert_{\mu,h}R^{-b_{h}}\}R^{-b_{h}}\ .
\end{align*}

Now if $R\geq 1$, then 
\begin{align*}
		&2^{b_{h}+1}\lVert C_n \rVert_{op}\lVert B \rVert_{\mu,h}+2^{b_{h}}\lVert C_n \rVert_{\mu,h}\lVert B \rVert_{op}+2^{2b_{h}+1}\lVert C_n \rVert_{\mu,h}\lVert B \rVert_{\mu,h}R^{-b_{h}}\leq \\
		&2^{b_{h}+1}\lVert C_n \rVert_{op}\lVert B \rVert_{\mu,h}+2^{b_{h}}\lVert C_n \rVert_{\mu,h}\lVert B \rVert_{op}+2^{2b_{h}+1}\lVert C_n \rVert_{\mu,h}\lVert B \rVert_{\mu,h}\ .
\end{align*}

Let $D_n=2^{b_{h}+1}\lVert C_n \rVert_{op}\lVert B \rVert_{\mu,h}+2^{b_{h}}\lVert C_n \rVert_{\mu,h}\lVert B \rVert_{op}+2^{2b_{h}+1}\lVert C_n \rVert_{\mu,h}\lVert B \rVert_{\mu,h}.$

If $0<R<1$, then $$\mu_{C_nB}(R)\leq\lVert C_nB\rVert_{op}\leq\lVert C_n\rVert_{op}\lVert B\rVert_{op}\leq\lVert C_n\rVert_{op}\lVert B\rVert_{op}\cdot R^{-b_{h}}.$$
So for all $R\in\mathbb R{>0},$ we have $\mu_{C_nB}(R)\leq \max\{D_n,\lVert C_n\rVert_{op}\lVert B\rVert_{op}\}\cdot R^{-b_{h}}.$
Hence, we get $\lVert C_nB\rVert_{\mu,h}\leq \max\{D_n,\lVert C_n\rVert_{op}\lVert B\rVert_{op}\}.$
Since we have $$\lim_{n\to\infty}D_n=\lim_{n\to\infty}\lVert C_n\rVert_{op}\lVert B\rVert_{op}=0,$$
we get $\lim_{n\to\infty}\lVert C_nB\rVert_{\mu,h}=0.$

Similarly, we can show that $\lim_{n\to\infty}\lVert A_n(B-B_n)\rVert_{\mu,h}=0.$
Thus, we get $$\lim_{n\to\infty}\lVert AB-A_nB_n\rVert_{\mu,h}=0.$$
Hence, we have $\lim_{n\to\infty}A_nB_n=AB.$
So $AB\in C^{pol}_{h}G.$
Thus, $C^{pol}_{h}G$ is an algebra.
\end{proof}

In the following, we are giving an estimate that is used in the proof of smoothness of $C^{pol}_{h}G.$
It can be proven by induction on $n.$

\begin{Lemma}\cite{Engel}
Given $A\in C^{pol}_{h}G$ and $n\in\mathbb{N},$ we have 
$$\mu_{(\operatorname{Id}-A)^{n}}(R)\leq\sum_{k=1}^{n-1} 5^k\lVert \operatorname{Id}-A\rVert_{op}^{n-1}\mu_{A}\Big( \frac{R}{2^k} \Big).$$
\end{Lemma}

In the following, we show that $C^{pol}_{h}G$ is a smooth subalgebra of $C^{*}_{r} G.$

\begin{Lemma}\cite{Engel}
$C^{pol}_{h}G$ is closed under holomorphic functional calculus.
\end{Lemma}

\begin{proof}

Given $A\in C^{pol}_{h}G$ with $\lVert \operatorname{Id}-A \rVert_{op}<\epsilon$, where $\epsilon=\frac{1}{2}\frac{1}{5\cdot 2^{b_{h}}}.$
We have

\begin{align*}
	\mu_{A^{-1}-\sum_{n=0}^{N} (\operatorname{Id}-A)^{n}}(R)	&\leq\sum_{n=N+1}^{\infty} \mu_{(\operatorname{Id}-A)^{n}}(R)\\
							&\leq\sum_{n=N+1}^{\infty} \sum_{k=1}^{n-1} 5^k\epsilon^{n-1}\mu_{A}\Big( \frac{R}{2^k} \Big)\\
							&=\sum_{k=1}^{N} \Big\{5^k\mu_{A}\Big( \frac{R}{2^k} \Big)\cdot\sum_{n=N+1}^{\infty}\epsilon^{n-1}\Big\}
							+\sum_{k=N+1}^{\infty} \Big\{5^k\mu_{A}\Big( \frac{R}{2^k} \Big)\cdot\sum_{n=k+1}^{\infty}\epsilon^{n-1}\Big\}\\
							&=\frac{\epsilon^N}{1-\epsilon}\sum_{k=1}^{N} 5^k\mu_{A}\Big( \frac{R}{2^k} \Big)
							+\frac{1}{1-\epsilon}\sum_{k=N+1}^{\infty} (5\epsilon)^k\mu_{A}\Big( \frac{R}{2^k} \Big)\\
							&\leq\frac{\epsilon^N\cdot \lVert A\rVert_{\mu,h}}{1-\epsilon}\cdot R^{-b_{h}}\cdot\sum_{k=1}^{N}(5\cdot 2^{b_h})^k
							+\frac{\lVert A\rVert_{\mu,h}}{1-\epsilon}\cdot R^{-b_{h}}\cdot\sum_{k=N+1}^{\infty}(5\cdot\epsilon\cdot 2^{b_h})^k\\
							&=\frac{\lVert A\rVert_{\mu,h}}{1-\epsilon}\cdot\bigg\{\epsilon^N\sum_{k=1}^{N}(5\cdot 2^{b_h})^k+\sum_{k=N+1}^{\infty}(\frac{1}{2})^k\bigg\}\cdot R^{-b_{h}}\ .
\end{align*}
Thus, we have $\lVert A^{-1}-\sum_{n=0}^{N} (\operatorname{Id}-A)^{n}\rVert_{\mu,h}\leq\frac{\lVert A\rVert_{\mu,h}}{1-\epsilon}\cdot\{\epsilon^N\sum_{k=1}^{N}(5\cdot 2^{b_h})^k+\sum_{k=N+1}^{\infty}(\frac{1}{2})^k\}.$
Now since $\lim_{N\to\infty}\{\epsilon^N\sum_{k=1}^{N}(5\cdot 2^{b_h})^k+\sum_{k=N+1}^{\infty}(\frac{1}{2})^k\}=0,$
we get $$\lim_{N\to\infty}\lVert A^{-1}-\sum_{n=0}^{N} (\operatorname{Id}-A)^{n}\rVert_{\mu,h}=0.$$

Hence we have $A^{-1}\in C^{pol}_{h}G.$
So $C^{pol}_{h}G$ is closed under holomorphic functional calculus by \cite[Lemma 1.2]{Sch} and \cite[Lemma 3.38]{Fgvb}.
\end{proof}




\subsection{Trace Functions}

In this section, we recall the trace function $\tau_h$ on $\mathbb{C}G$ corresponding to an element $h\in G.$
If $h\in G^{pol},$ then we extend this trace to a trace $\widetilde{\tau}_{h}$ on $C^{pol}_{h}G.$ 

In the following, we are recalling the classical trace on $\mathbb{C}G$ corresponding to an element $h\in G.$
\begin{Def}
For all $h\in G,$ let $\tau_{h} : \mathbb{C}G\rightarrow\mathbb{C}$ be defined as:
$$\tau_{h}(\sum_{g\in G} a_{g}.g) := \sum_{g\in C(h)} a_{g},$$
where $C(h)$ is the conjugacy class of $h.$
\end{Def}

It is clear that $\tau_{h}$ is a trace function on $\mathbb{C}G$ \cite{St}.

In the following, we introduce a notation so that, we can use operators as matrices.

\begin{Def}
Given $A\in\mathcal{B}(\ell^2 G)$, define $A(g,f):=(A \delta_f)(g)$ for all $g,f\in G,$
where 
\[\delta_f(k)= \begin{cases} 
      1 & \text{if}\ k=f \\
      0 & \text{otherwise}\ .\\ 
   \end{cases}
\]
\end{Def}

The following equivariance property is used in the proof that liftings 
$$\widetilde{\tau}_h :C^{pol}_{h} G\rightarrow \mathbb{C}$$
are trace functions. It can be shown with a direct calculation.
\begin{Lemma}
 Given $A\in C^{*}_{r}G$ and $g,f,h\in G$, we have $A(g,f)=A(gh,fh).$
\end{Lemma}

In the following, we define a lifting of the classical trace function $\tau_h:\mathbb{C}G\rightarrow\mathbb{C}.$

\begin{Def}
For each $h\in G^{pol},$ define $\widetilde{\tau}_{h} : C^{pol}_{h}G \rightarrow\mathbb{C}$ as,
$$\widetilde{\tau}_{h}(A)=\sum_{g\in C(h)} A(g,e)\ .$$

\end{Def}

In the following, we prove an inequality that we use in the proof of the Theorem ~\ref{well}.
\begin{Lemma}\label{connect}
For all $A\in\mathcal{B}(\ell^2 G)\text{ and }R\in\mathbb R_{>0}$ we have
$$\bigg(\sum_{\lVert g \rVert_{w} > R} |A(g,e)|^2\bigg)^{\frac{1}{2}}\leq\mu_{A}(R)\ .$$
\end{Lemma}

\begin{proof}
For all $A\in\mathcal{B}(\ell^2 G)\text{ and }R\in\mathbb R_{>0}$ we have
\begin{align*}
	\bigg(\sum_{\lVert g \rVert_{w} > R} |A(g,e)|^2\bigg)^{\frac{1}{2}}
		&=\bigg(\sum_{g\in G\setminus B_{R}(\{e\})} |(A \delta_e)(g)|^2\bigg)^{\frac{1}{2}}\\
		&=\lVert A \delta_{e}\rVert_{G\setminus B_{R}(\operatorname{Supp}\ \delta_{e})}\\
		&\leq \mu_{A}(R) \lVert \delta_{e}\rVert\\
		&=\mu_{A}(R).\qedhere
\end{align*}
\end{proof}
Since we defined $\widetilde{\tau}_{h}$ to be a sum over the (possibly infinite) set $C(h)$, we need to prove that the sum converges.
In the following, we show that the sum absolutely converges.
\begin{Th} \label{well}
$\widetilde{\tau}_{h} : C^{pol}_{h}G \rightarrow\mathbb{C}$ is well defined and continuous.
\end{Th}

\begin{proof}
Given $A\in C^{pol}_{h}G,$ we have 2 cases:

If $h=e$, then $|\widetilde{\tau}_{h}(A)|=|A(e,e)|<\infty.$
If $h\neq e$, then we have
\begin{align*}
	|\widetilde{\tau}_{h}(A)|
				&\leq\sum_{l=1}^{\infty}\sum_{g\in C_{l}(h)} |A(g,e)|	&\text{(since $h\neq e,$ $l$ starts from $1$)}\\
				&\leq\sum_{l=1}^{\infty} \sqrt{n_{h,l}}\cdot \bigg(\sum_{g\in C_{l}(h)} |A(g,e)|^{2}\bigg)^{\frac{1}{2}}	&\text{(Cauchy-Schwarz inequality)}\\
				&\leq\sum_{l=1}^{\infty} \sqrt{n_{h,l}}\cdot \bigg(\sum_{\lVert g \rVert_{w} > (l-\frac{1}{2})} |A(g,e)|^{2}\bigg)^{\frac{1}{2}}\\
				&\leq\sum_{l=1}^{\infty} \sqrt{n_{h,l}}\cdot \mu_{A}\bigg(l-\frac{1}{2}\bigg)\ &\text{(Lemma\ \ref{connect})}\\
				&\leq\sum_{l=1}^{\infty} \sqrt{C_{h}}\cdot l^{\frac{d_{h}}{2}}\cdot\lVert A \rVert_{\mu,h}\cdot \bigg(l-\frac{1}{2}\bigg)^{-b_{h}}\\
				&\leq C\cdot \sqrt{C_{h}}\cdot \lVert A \rVert_{\mu,h}\cdot \sum_{l=1}^{\infty} l^{-2},\text{ for some } C\in\mathbb R_{>0}.\	&(b_h\geq\frac{d_{h}}{2}+2)\\
				&< \infty\ .
\end{align*}
Hence $\widetilde{\tau}_{h} : C^{pol}_{h}G \rightarrow\mathbb{C}$ is well-defined and continuous.
\end{proof}

In the following, we show indeed $\widetilde{\tau}_{h} : C^{pol}_{h}G \rightarrow\mathbb{C}$ is a trace function extending the classical trace function
$\tau_{h} : \mathbb{C}G\rightarrow\mathbb{C}.$

\begin{Th}
For all $h\in G^{pol} ,\ \widetilde{\tau}_{h}$ is a trace function on $C^{pol}_{h}G$ extending $\tau_{h}.$
\end{Th}

\begin{proof}
The only nontrivial part is to show that $\forall A,B\in C^{pol}_{h}G$ we have $\widetilde{\tau}_{h}(AB)=\widetilde{\tau}_{h}(BA):$

Given $A,B\in C^{pol}_{h}G$, we have \footnote{We have absolute convergence in the sums (by the proof of Lemma $~\ref{well}$). So we can change the order of summation as we want.}
\begin{align*}
\widetilde{\tau}_{h}(AB)	&=\sum_{f\in G}\sum_{g\in C(h)}A(gf^{-1},e)B(f,e)\\
				&=\sum_{f\in G}\sum_{k\in C(h)}A(f^{-1}k,e)B(f,e)\	\	&(k=fgf^{-1})\\
				&=\sum_{k\in C(h)}\sum_{f\in G}A(f^{-1}k,e)B(f,e)\\
				&=\sum_{k\in C(h)}\sum_{l\in G}A(l,e)B(kl^{-1},e)\	\	&(l=f^{-1}k)\\
				&=\widetilde{\tau}_{h}(BA).\qedhere
\end{align*}
\end{proof}



\section{proof of Theorem ~\ref{mainpol} and its applications}\label{polynomialproof}

In this section, we present the proof of Theorem ~\ref{mainpol} and apply the result to derive lower bounds for the groups $S(M)$ and $P(M).$

\begin{proof}[Proof of Theorem ~\ref{mainpol}]
Since, for all $g,h\in G^{fin},$ we have 
$$g\simfin  h\implies [p_g]=[p_h]\in K_{0}^{fin}(C^{*}G),$$
 we get $\operatorname{rank}(K_{0}^{fin}(C^{*}G))\leq\mathcal{F}_G.$
Using the surjection $K_{0}^{fin}(C^{*}G)\twoheadrightarrow K_{0}^{fin}(C^{*}_{r}G),$ we can conclude that $\operatorname{rank}(K_{0}^{fin}(C^{*}_{r}G))\leq\operatorname{rank}(K_{0}^{fin}(C^{*}G)).$

For the rest, it suffices to prove that $S= G^{pol}\cap G^{fin}$ and $\mathcal{A}=C^{pol}_{S}G:=\bigcap_{h\in S}C^{pol}_{h}G$ 
satisfies the hypotheses of the Theorem ~\ref{framework}.

Since $C^{pol}_{h}G's$ are smooth subalgebras of $C^{*}_{r}G$ containing $\mathbb{C}G,$ we get $C^{pol}_{S}G$ is a smooth subalgebra of $C^{*}_{r}G$ containing 
$\mathbb{C}G.$

Since $G^{pol}\cap G^{fin}\subseteq G^{pol},$ for all $h\in G^{pol}\cap G^{fin},$ $\tau_{h}:\mathbb{C}G\rightarrow\mathbb{C}$ has a lift 
$$\widetilde{\tau}_{h}:C^{pol}_{h}G\rightarrow\mathbb{C}.$$
Therefore, for all $h\in S,$ the trace function $\tau_{h}:\mathbb{C}G\rightarrow\mathbb{C}$ has a lift 
$$\widetilde{\tau}_{h}:C^{pol}_{S}G\rightarrow\mathbb{C},$$
which is also a trace function.

Hence, we get $\operatorname{rank}(K^{fin}_{0}(C^{*}_{r} G))\geq |S/_{\simfin} |=|(G^{pol}\cap G^{fin})/_{\simfin} |=\mathcal{F}^{pol}_{G}.$
For the assembly map $\mu:K^{G}_{0}(EG)\rightarrow K_{0}(C^{*}G),$ we have $\operatorname{\operatorname{Im}}\mu\cap K^{fin,pol}_{0}(C^{*}G)=\{0\}.$
\end{proof}

\subsection{Applications}\label{polapply}
In this subsection, we combine the results from Weinberger and Yu \cite{Yu} and Theorem ~\ref{mainpol} to derive lower bounds for the ranks of the structure group and the group of positive scalar curvature metrics of manifolds.

Given a compact oriented manifold $M,$ we define the structure group $S(M)$ of $M$ to be the abelian group generated by the equivalence classes of elements of the form $(f,M\textprime),$
where  $M\textprime$ is a compact oriented manifold and $f:M\textprime\rightarrow M$ is an orientation preserving homotopy equivalence. We say $(f_1,M_1)$ is equivalent to $(f_2,M_2)$ if there exists an h-cobordism $(W;M_1,M_2)$ and a homotopy equivalence $F:W\rightarrow M$ such that restrictions of $F$ to $M_1$ and $M_2$ gives $f_1$ and $f_2$ respectively \cite[Definition 1.14]{Ranicki}.

We have the following result about the structure group $S(M)$ of a compact oriented manifold $M$ from Weinberger and Yu.
\begin{Th}\label{structuregroup}\cite{Yu}
Let $M$ be a compact oriented manifold with dimension $4k-1$ $(k>1).$ Suppose $\pi_1(M)=G$ and $g_1,\cdots, g_n$ be finite order elements in $G$ such that $g_i\neq e$
for all $i$ and $\{[p_{g_1}],\cdots,[p_{g_n}]\}$ generates an abelian group of $K_0(C^{*}G)$ with rank n. Suppose that any nonzero element in the abelian subgroup of 
$K_0(C^{*}G)$ generated by $\{[p_{g_1}],\cdots,[p_{g_n}]\}$ is not in the image of the map $\mu : K_{0}^{G}(EG)\rightarrow K_{0}(C^{*}G)$, then the rank of the structure group $S(M)$ is greater than or equal to n.
\end{Th}

Now we combine the previous result about $S(M)$ with Theorem ~\ref{mainpol}, where the lower bound is in terms of $\mathcal{F}^{pol}_G.$

\begin{Cor}\label{structuregroupour}
For a compact oriented manifold $M$ with dimension $4k-1$ $(k>1),$
the rank of the structure group $S(M)$ is greater than or equal to $\mathcal{F}^{pol}_G-1,$
where $G=\pi_1(M).$
\end{Cor}

\begin{proof}
We have $\operatorname{rank}(K^{fin,pol}_{0}(C^{*}G))\geq\mathcal{F}^{pol}_G$ by the proof of Theorem ~\ref{mainpol} and, we have $\operatorname{Im}(\mu )\cap K_{0}^{fin,pol}(C^{*}G)=\{0\}.$
Since we have $e\in G^{fin}\cap G^{pol},$ we get the rank of the structure group $S(M)$ is greater than or equal to $\mathcal{F}^{pol}_G -1.$
\end{proof}

Let $r_{fin}(G)$ be the rank of the abelian group $K^{fin}_{0}(C^{*}G)$ generated
by $[p_g]$ for all finite order elements $g\in G.$ Here $g$ is allowed to be the
identity element $e.$ So we have $r_{fin}(G)=\operatorname{rank}(K^{fin}_{0}(C^{*}G))\geq\mathcal{F}^{pol}_G.$

If a compact smooth spin manifold $M$ has a positive scalar curvature metric and the dimension of $M$ is greater than or equal to $5,$
then we define (roughly) $P(M)$ to be the abelian group of equivalent classes of all positive scalar curvature metrics on M. For a more precise definition, we refer to 
\cite[Section 4]{Yu}.

We have the following result about the group $P(M)$ from Weinberger and Yu.
\begin{Th}\label{positivescalar}\cite{Yu}
  \begin{enumerate}
    \item
    Let $M$ be a compact smooth spin manifold with a positive scalar curvature metric and dimension $2k-1$ $(k>2).$
    The rank of the abelian group $P(M)$ is greater than or equal to $r_{fin}(G)-1.$

    \item
    Let $M$ be a compact smooth spin manifold with a positive scalar curvature metric and dimension $4k-1$ $(k>1).$
    The rank of the abelian group $P(M)$ is greater than or equal to $r_{fin}(G).$
  \end{enumerate}
\end{Th}

In the following, we combine the previous result about $P(M)$ with Theorem ~\ref{mainpol}. 
The lower bounds are in terms of $\mathcal{F}^{pol}_G.$

\begin{Cor}\label{positivescalarour}
Let $M$ be a compact smooth spin manifold with a positive scalar curvature metric and let $G=\pi_1(M).$
  \begin{enumerate}
    \item
    If $M$ has dimension $2k-1$ $(k>2),$ then the rank of the abelian group $P(M)$ is greater than or equal to $\mathcal{F}^{pol}_G -1.$

    \item
    If $M$ has dimension $4k-1$ $(k>1),$ then the rank of the abelian group $P(M)$ is greater than or equal to $\mathcal{F}^{pol}_G.$
  \end{enumerate}

\end{Cor}

\begin{proof}
We have $r_{fin}(G)\geq\mathcal{F}^{pol}_G.$
\end{proof}

\section{Polynomially Full Groups}
In this section, we define the class of polynomially full groups.
We show that subgroups, products, and finite extensions of polynomially full groups are also polynomially full.
For polynomially full group $G,$ we show that
$$K^{fin}_{0}(C^{*}_{r}G)\cong K^{fin}_{0}(C^{*}G)\cong\bigoplus_{i=1}^{\mathcal{F}_G}\mathbb{Z}.$$
The class of polynomially full groups includes trivially all finite groups and finitely generated torsion-free groups. We show that it also includes all finitely generated virtually nilpotent groups.
At the end of the section, we derive formulas for the number $\mathcal{F}_G,$ where $G$ is finitely generated abelian group, dihedral group, or symmetric group.

For a finitely generated group $G$ with a finite generating set $S,$
we denote the word-length norm by $\lVert .\rVert,$ $\lVert .\rVert_{S},$ or $\lVert .\rVert_{G}.$
For $g\in G,$ we denote the conjugacy class of $g$ in $G$ by $C^{G}(g).$
We denote the set of elements in the conjugacy class of $g$ with length (with respect to $S$) $l$
by $C^{G}_{l}(g).$

In the following, we give two equivalent conditions for a group. We use these conditions to define the class of polynomially full groups.
\begin{Prop}\label{equivalentconditions}
 For a finitely generated group $G$ the following are equivalent:
 \begin{enumerate}
  \item $G^{fin}\subseteq G^{pol}.$
  \item For all $g\in G^{fin}$ there exists $h\in G^{fin}\cap G^{pol}$ such that $g\simfin h.$
 \end{enumerate}
\end{Prop}

\begin{proof}
 $(1)\implies(2):$ Obvious.\\
 $(2)\implies(1):$ Given $g\in G^{fin},$ there exists $h\in G^{fin}\cap G^{pol}$ such that $g\simfin h.$
 So there exist $a,b\in\mathbb{N}$ and $f\in G$ such that $g^{a}=fhf^{-1}$ and $h^{b}\in C^{G}(g).$
 Define $A_{l}=\{\alpha^{a}\ :\ \alpha\in C^{G}(g),\ \lVert\alpha\rVert=l\}.$
 Define $\operatorname{i}:C^{G}_{l}(g)\rightarrow A_{l}$ by 
 $$\operatorname{i}(\alpha)=\alpha^{a}$$
 and $\operatorname{j}:A_{l}\rightarrow C^{G}_{l}(g)$ by 
 $$\operatorname{j}(\beta)=\beta^{b}.$$
 It is easy to see that, $\operatorname{j}\circ\operatorname{i}=\operatorname{id}_{C^{G}_{l}(g)}$
 and $\operatorname{i}\circ\operatorname{j}=\operatorname{id}_{A_{l}}.$
 Hence, we have 
 $$|C^{G}_{l}(g)|=|A_{l}|.$$
 
 Now, let $B_{l}=\{\beta\in C^{G}(h)\ :\ \lVert\beta\rVert\leq a\cdot l\}.$
 We show that $A_{l}\subseteq B_{l}:$
 
 Given $\omega\in A_{l},$ there exists $\alpha\in C^{G}(g)$ with $\lVert\alpha\rVert=l$ and $\omega=\alpha^{a}.$
 Since $\alpha\in C^{G}(g),$ there exists $\gamma\in G$ such that $\alpha=\gamma g\gamma^{-1}.$
 So $\omega=\alpha^{a}=\gamma g^{a}\gamma^{-1}=\gamma fhf^{-1}\gamma^{-1}.$
 Hence, $\omega\in C^{G}(h)$ with $\lVert\omega\rVert=\lVert\alpha^{a}\rVert\leq a\cdot\lVert\alpha\rVert=a\cdot l.$
 Thus, $\omega\in B_{l}.$
 Therefore, we have $A_{l}\subseteq B_{l}.$
 
 So we get $|C^{G}_{l}(g)|=|A_{l}|\leq |B_{l}|.$
 Since $h\in G^{pol},$ $|B_{l}|$ is bounded from above by a polynomial of $l.$
 Thus, we get $g\in G^{pol}.$
\end{proof}

\begin{Def}
 Let $G$ be a finitely generated group. We say that $G$ is polynomially full, if it satisfies the conditions from Proposition ~\ref{equivalentconditions}.
\end{Def}
Obviously finite groups and finitely generated torsion-free groups are polynomially full.

The following result is the motivation behind the definition of polynomially full groups.
\begin{Th}
 For a polynomially full group $G$ we have 
 $$K^{fin}_{0}(C^{*}_{r}G)\cong K^{fin}_{0}(C^{*}G)\cong\bigoplus_{i=1}^{\mathcal{F}_G}\mathbb{Z}.$$
\end{Th}

\begin{proof}
Let $G$ be a polynomially full group. 
We have
\begin{align*}
 \mathcal{F}_G	&= |G^{fin}/_{\simfin} |\\
		&= |(G^{fin}\cap G^{pol})/_{\simfin} |\\
		&= \mathcal{F}^{pol}_G
\end{align*}

So we have $\bigoplus_{i=1}^{\mathcal{F}_G}\mathbb{Z}=\bigoplus_{i=1}^{\mathcal{F}^{pol}_G}\mathbb{Z}.$ Let $F=(G^{fin}\cap G^{pol})/_{\simfin} .$
Hence, we have $\bigoplus_{i=1}^{\mathcal{F}^{pol}_G}\mathbb{Z}\cong \bigoplus_{[g]\in F}\mathbb{Z}.$
Now, define 
$$\phi:\bigoplus_{[g]\in F}\mathbb{Z}\rightarrow K^{fin}_{0}(C^{*}_{r}G)$$ as 
$$\phi(\delta_{[g]})=[p_g],$$
where $\delta_{[g]}$ is the canonical basis element corresponding to the $[g]$-component.
Since $g\simfin h$ implies $[p_g]=[p_h],$ $\phi$ is a well defined homomorphism.
Recall that, $K^{fin}_{0}(C^{*}_{r}G)$ is generated by the set $\{[p_g]:g\in G^{fin}\}.$ Since $G$ is polynomially full, we have 
$$\{[p_g]:g\in G^{fin}\}=\{[p_g]:g\in G^{fin}\cap G^{pol}\}.$$
Hence, $\phi$ is surjective.

On the other hand, by the proof of the Theorem ~\ref{framework}, $\phi(\delta_{[g]})=[p_g]'s$ are $\mathbb{Z}$-linearly independent. Thus, $\phi$ is injective. So we have $\bigoplus_{[g]\in F}\mathbb{Z}\cong K^{fin}_{0}(C^{*}_{r}G).$ Therefore, we get $K^{fin}_{0}(C^{*}_{r}G)\cong\bigoplus_{i=1}^{\mathcal{F}_G}\mathbb{Z}.$
The isomorphism $K^{fin}_{0}(C^{*}G)\cong\bigoplus_{i=1}^{\mathcal{F}_G}\mathbb{Z}$ can be shown similarly.
\end{proof}

In the following, we show that finite extensions and images of polynomially full groups under homomorphisms with finite kernels are also polynomially full.
\begin{Prop}
 Let $F$ be a finite group and let $G,H$ be finitely generated groups.
 If we have a short exact sequence
 \[
    \begin{tikzcd}
      &1\arrow{r} &F\arrow{r}{\alpha} &G\arrow{r}{\beta} &H\arrow{r} &1,
    \end{tikzcd}
 \]
 then $G$ is polynomially full if and only if $H$ is polynomially full.
 In this case, we have $\mathcal{F}_H\leq\mathcal{F}_G.$
\end{Prop}

\begin{proof}
 Without loss of generality, assume that the generating set of $H$ is in the image of the generating set of $G$ under $\beta.$
 
 Assume $G$ is polynomially full. 
 Given $h\in H^{fin},$ since $\beta$ is onto, there exists $g\in G$ with $\beta(g)=h.$
 Since $F$ is finite, we get $g\in G^{fin}.$
 Now, since $G$ is polynomially full, $\exists g'\in G^{fin}\cap G^{pol}$ such that $g\simfin g'.$
 Now, let $h'=\beta(g')\in H.$ Since $g'\in G^{fin},$ we have $h'\in H^{fin}.$
 Since $g\simfin g',$ we have $h\simfin h'.$
 Now, since $F$ is finite, there exists $R\in\mathbb{N}$ such that $\operatorname{ker}\beta=\operatorname{\operatorname{Im}}\alpha\subseteq\operatorname{B}_{R}(e_{G}),$
 where $\operatorname{B}_{R}(e_{G})$ is the closed ball around the identity element of $G$ with radius $R.$
 
 Since, $\cup_{l=1}^{N} C^H_{l}(h')\subseteq\beta(\cup_{l=1}^{N+R} C^G_{l}(g')),$ we have $|\cup_{l=1}^{N} C^H_{l}(h')|\leq |\cup_{l=1}^{N+R} C^G_{l}(g')|,$ 
 and right hand side is bounded by some polynomial of $N+R$ (hence by a polynomial of $N$).
 So we get $h'\in H^{pol}.$
 Hence, $h'\in H^{fin}\cap H^{pol}$ with $h'\simfin h.$
 Therefore, $H$ is polynomially full.
 
 For the converse, assume $H$ is polynomially full. Given $g\in G^{fin},$ let $h=\beta(g)\in H^{fin}.$
 Since $H$ is polynomially full, we have $h\in H^{pol}.$
 Now we have 
 $$\beta(\cup_{i=0}^{l}C^G_{i}(g))\subseteq\cup_{i=0}^{l}C^H_{i}(h).$$
 Hence, we get $|\cup_{i=0}^{l}C^G_{i}(g)|\leq |F|\cdot |\beta(\cup_{i=0}^{l}C^G_{i}(g))|\leq |F|\cdot |\cup_{i=0}^{l}C^H_{i}(h)|.$
 Since $h\in H^{pol}$ and $F$ is finite, right hand side is bounded by a polynomial of $l.$
 Thus, we get $g\in G^{pol}.$
 Therefore, $G$ is polynomially full.
 
 Now, given $g_{1},g_{2}\in G^{fin},$ $g_{1}\simfin g_{2}$ (in $G$) implies $\beta(g_{1})\simfin \beta(g_{2})$ (in $H$).
 Thus, we have $\mathcal{F}_{H}=|H^{fin}/_{\simfin} |=|\beta(G^{fin})/_{\simfin} |\leq |G^{fin}/_{\simfin} |=\mathcal{F}_{G}.$
\end{proof}

In the following, we prove that the property of being polynomially full is inherited to finitely generated subgroups.
\begin{Prop}\label{polfullinherited}
 Let $G$ be a finitely generated group. Let $H$ be a finitely generated subgroup of $G.$
 If $G$ is polynomially full, then $H$ is also polynomially full.
\end{Prop}

\begin{proof}
 Let $S$ and $T$ be finite generating sets of $G$ and $H$ respectively.
 Without loss of generality, we can assume that $T\subseteq S.$
 
 Now, given $h\in H^{fin},$ we have $h\in G^{fin}.$
 Since $G$ is polynomially full, we have $h\in G^{pol}.$
  
 It is easy to see that $C^{H}(h)\subseteq C^{G}(h).$
 Now, for all $f\in H,$ we have $\lVert f\rVert_{T}\geq\lVert f\rVert_{S}.$
 Hence, $C^{H}_{l}(h)\subseteq\cup_{i=0}^{l}C^{G}_{i}(h).$
 Thus, we have $|C^{H}_{l}(h)|\leq|\cup_{i=0}^{l}C^{G}_{i}(h)|.$
 Since $h\in G^{pol},$ right hand side is bounded by a polynomial of $l.$
 Therefore, we get $h\in H^{pol}.$
 Hence, $H$ is polynomially full.
\end{proof}

In the following, we show that the class of polynomially full groups is closed under taking direct products.
\begin{Prop}
 Let $G$ and $H$ be finitely generated groups. Then $G$ and $H$ are polynomially full if and only if $G\times H$ is polynomially full.
\end{Prop}

\begin{proof}
 Let $S$ and $T$ be finite generating sets for $G$ and $H$ respectively. Then, 
 $$W:=S\times\{e_H\}\cup \{e_G\}\times T$$
 is a finite generating set for $G\times H.$
 Let $\lVert .\rVert_G,\ \lVert .\rVert_H,$ and $\Vert .\rVert_{(G\times H)}$ be the word-length norms on $G,H,$ and $G\times H$ respectively, corresponding to the generating sets $S,T,$ and $W$ respectively.
 
 Assume $G$ and $H$ are polynomially full.
 Given $(g,h)\in (G\times H)^{fin},$ we have $g\in G^{fin}$ and $h\in H^{fin}.$
 Since $G$ and $H$ are polynomially full, we have $g\in G^{pol}$ and $h\in H^{pol}.$
 It is not hard to see that $\Vert (g',h')\rVert_{(G\times H)}=\lVert g'\rVert_G+\lVert h'\rVert_H$ for all $g'\in G$ and $h'\in H.$
 So we have 
 $$|C^{G\times H}_n((g,h))|=\sum_{i=0}^{n} |C^G_i(g)|\cdot |C^H_{n-i}(h)|.$$
 All the terms in the sum are bounded by polynomials of $n.$ Thus, the sum is bounded by a polynomial of $n.$
 Hence, $(g,h)\in (G\times H)^{pol}.$
 Therefore $G\times H$ is polynomially full.
 
 Converse follows from Proposition ~\ref{polfullinherited}.
\end{proof}

In the following, we give a sufficient condition for a group to be polynomially full.
Recall that a subset of a group is said to grow polynomially if the number of elements in the intersection of the subset with the closed ball of radius $l$ centered around the identity element is bounded by a 
fixed polynomial of $l.$
\begin{Lemma}\label{torsiongrowspol}
 Let $G$ be a finitely generated group. If $G^{fin}$ grows polynomially, then $G$ is polynomially full. 
\end{Lemma}

\begin{proof}
 For all $g\in G^{fin}$ we have $C^G(g)\subseteq G^{fin}.$ Since $G^{fin}$ grows polynomially, $C^G(g)$ also grows polynomially.
 Hence, $g\in G^{pol}.$ Therefore, $G$ is polynomially full.
\end{proof}

Wolf \cite[Theorem 3.11]{Wolf} showed that for finitely generated group $\Sigma$ and a subgroup $\Gamma$ of finite index, we have that
  \begin{enumerate}
      \item
      $\Gamma$ is finitely generated, and

      \item
      if $\Gamma$ has polynomial growth, then $\Sigma$ also has polynomial growth.
  \end{enumerate}
  
He also showed in \cite[Theorem 3.2]{Wolf} that, if $\Gamma$ is a finitely generated nilpotent group, then $\Gamma$ has polynomial growth.

Gromov \cite{Gromov} showed that if a finitely generated group $\Gamma$ has polynomial growth, then it is virtually nilpotent.
Recall that a group is called virtually nilpotent, if it contains a nilpotent subgroup of finite index.

In the following, we show that the class of polynomially full groups includes finitely generated virtually nilpotent groups.
\begin{Cor}\label{virtuallynilpotent}
 Let $G$ be a finitely generated group. If $G$ is virtually nilpotent, then $G$ is polynomially full.
\end{Cor}

\begin{proof}
 Let $H$ be a nilpotent subgroup of $G$ with finite index. By \cite[Theorem 3.11]{Wolf}, $H$ is also finitely generated.
 So by \cite[Theorem 3.2]{Wolf}, $H$ has polynomial growth.
 Hence, by \cite[Theorem 3.11]{Wolf}, $G$ has polynomial growth.
 Thus, $G^{fin}$ also has polynomial growth.
 Therefore, $G$ is polynomially full by Lemma ~\ref{torsiongrowspol}.
\end{proof}

In the following propositions, we derive formulas for the number $\mathcal{F}_G$ for some polynomially full groups. Recall that when $G$ is polynomially full, $\mathcal{F}_G$ is the rank of the 
free abelian groups $K^{fin}_{0}(C^{*}_{r}G)\cong K^{fin}_{0}(C^{*}G)\cong\bigoplus_{i=1}^{\mathcal{F}_G}\mathbb{Z}.$

In the following, we give a formula for $\mathcal{F}_{G}$ for a finite abelian group $G.$
\begin{Prop}\label{finiteabelian}
 For $G=\mathbb{Z}/n_{1}\times\cdots\times\mathbb{Z}/n_{k},$ we have the following formula 
 $$\mathcal{F}_G=\sum_{d_{1}|n_{1}}\cdots\sum_{d_{k}|n_{k}}\frac{\phi(d_{1})\cdots\phi(d_{k})}{\phi(\operatorname{lcm}(d_{1},\cdots ,d_{k}))},$$
 where $\phi$ denotes Euler's totient function, $\operatorname{lcm}$ denotes the least common multiple function,
 and the sums run over positive divisors $d_i$'s of $n_i$'s.
\end{Prop}

\begin{proof}
 Let's define an equivalence relation $\sim$ on $G=\mathbb{Z}/n_{1}\times\cdots\times\mathbb{Z}/n_{k},$ which is coarser than $\simfin :$
 
 We say $(x_{1},\cdots ,x_{k})\sim (x'_{1},\cdots ,x'_{k})$ if and only if, for all $i\in\{1,\cdots ,k\}$ $x_{i}\simfin  x'_{i}$ in $\mathbb{Z}/n_{i}.$
 It is easy to see that, $\sim$ is an equivalence relation on $G.$
 
 Since homomorphic images of equivalent elements are equivalent, by looking at the projections to components, we can conclude that 
 $(x_{1},\cdots ,x_{k})\simfin (y_{1},\cdots ,y_{k})$ implies $(x_{1},\cdots ,x_{k})\sim (y_{1},\cdots ,y_{k}),$ 
 for all $(x_{1},\cdots ,x_{k}), (y_{1},\cdots ,y_{k})\in G.$
 So $\sim$ is coarser than $\simfin .$
 
 For all $d_{1},\cdots ,d_{k}\in\mathbb{N}$ with $d_{1}|n_{1},\cdots ,d_{k}|n_{k},$ define 
 $$G_{d_{1},\cdots ,d_{k}}:=\{(x_{1},\cdots ,x_{k})\in G: \operatorname{gcd}(n_{i},x_{i})=\frac{n_{i}}{d_{i}}\text{ for }i\in\{1,\cdots ,k\}\}.$$
 It is easy to see that $G_{d_{1},\cdots ,d_{k}}=[(x_{1},\cdots ,x_{k})]_{\sim}$ for all $(x_{1},\cdots ,x_{k})\in G_{d_{1},\cdots ,d_{k}},$
 where $[(x_{1},\cdots ,x_{k})]_{\sim}$ is the equivalence class of the element $(x_{1},\cdots ,x_{k})$ with respect to $\sim.$
 
 Now, for all $(x_{1},\cdots ,x_{k})\in G_{d_{1},\cdots ,d_{k}},$ we have 
 \begin{align*}
  |[(x_{1},\cdots ,x_{k})]_{\simfin }|	&= \phi(\operatorname{order}((x_{1},\cdots ,x_{k})))\\
						&= \phi(\operatorname{lcm}(\operatorname{order}(x_{1}),\cdots ,\operatorname{order}(x_{k})))\\
						&= \phi(\operatorname{lcm}(d_{1},\cdots ,d_{k}))
 \end{align*}
 and we have $|G_{d_{1},\cdots ,d_{k}}|=\phi(d_{1})\cdots\phi(d_{k}).$
 Hence, we have $|G_{d_{1},\cdots ,d_{k}}/_{\simfin} |=\frac{\phi(d_{1})\cdots\phi(d_{k})}{\phi(\operatorname{lcm}(d_{1},\cdots ,d_{k}))}.$
 Thus, we get 
 \begin{align*}
  \mathcal{F}_{G}	&= |G^{fin}/_{\simfin} |\\
			&= |G/_{\simfin} |\\
			&= \sum_{d_{1}|n_{1}}\cdots\sum_{d_{k}|n_{k}} |G_{d_{1},\cdots ,d_{k}}/_{\simfin} |\\
			&= \sum_{d_{1}|n_{1}}\cdots\sum_{d_{k}|n_{k}} \frac{\phi(d_{1})\cdots\phi(d_{k})}{\phi(\operatorname{lcm}(d_{1},\cdots ,d_{k}))}.
 \end{align*}

\end{proof}

In the following, we give a formula for $\mathcal{F}_{G}$ for a finitely generated abelian group $G.$
\begin{Cor}
 For $G=\mathbb{Z}/n_{1}\times\cdots\times\mathbb{Z}/n_{k}\times\mathbb{Z}^{m},$ we have the following formula 
 $$\mathcal{F}_G=\sum_{d_{1}|n_{1}}\cdots\sum_{d_{k}|n_{k}}\frac{\phi(d_{1})\cdots\phi(d_{k})}{\phi(\operatorname{lcm}(d_{1},\cdots ,d_{k}))},$$
 where the sums run over positive divisors $d_i$'s of $n_i$'s.
\end{Cor}

\begin{proof}
 Let $H=\mathbb{Z}/n_{1}\times\cdots\times\mathbb{Z}/n_{k}.$ Since $G$ is abelian and $G^{fin}\cong H,$ result follows from Proposition ~\ref{finiteabelian}.
\end{proof}

\begin{Rem}
 If we take $G=\mathbb{Z}/n,$ then the above formula tells that $\mathcal{F}_G$ is equal to the number of positive divisors of $n.$
\end{Rem}

In the following, we give a formula for $\mathcal{F}_{G}$ for a dihedral group $G.$
\begin{Prop}
 For $G=D_{n},$ we have 
 \[\mathcal{F}_{G}= \begin{cases} 
      \mathcal{F}_{\mathbb{Z}/n}+1 & \text{if } n \text{ is odd}\\
      \mathcal{F}_{\mathbb{Z}/n}+2 & \text{otherwise ,}\\ 
   \end{cases}
\]
where $D_{n}$ is the dihedral group of order $2n.$
\end{Prop}

\begin{proof}
 Let $x,y$ be the generators of $D_{n}$ with $x^2=y^n=(xy)^2=1.$
 For all $a,b\in\mathbb{Z},$ we have $(xy^{a})\cdot y^{b}\cdot (xy^{a})^{-1}=(xy^{a})\cdot y^{b}\cdot y^{-a}x=x^2y^{-b}=y^{-b},$
 and $y^{a}\cdot y^{b}\cdot y^{-a}=y^{b}.$
 So for all $a\in\mathbb{Z},$ we get $[y^{a}]_{\simfin }\subseteq \{1,y,\cdots ,y^{n-1}\},$
 where $[y^{a}]_{\simfin }$ denotes the equivalence class of $y^{a}$ in $D_{n}.$
 
 Now let's show that $y^{a}\simfin  y^{b}$ if and only if $\operatorname{gcd}(n,a)=\operatorname{gcd}(n,b):$
 
 For the forward direction, we have 
 \begin{align*}
		y^{a}\simfin  y^{b}	&\implies \operatorname{order}(y^{a})=\operatorname{order}(y^{b})\\
													&\implies \operatorname{gcd}(n,a)=\operatorname{gcd}(n,b).
 \end{align*}
 For the converse, assume we have $d=\operatorname{gcd}(n,a)=\operatorname{gcd}(n,b)$ for some $d\in\mathbb{N}.$
 So we get $\operatorname{order}(y^{a})=\operatorname{order}(y^{b})=\frac{n}{d}$ and 
 $\operatorname{gcd}(\frac{n}{d},\frac{a}{d})=\operatorname{gcd}(\frac{n}{d},\frac{b}{d})=1.$
 Hence, there exists $c\in\mathbb{N}$ such that $c\cdot\frac{a}{d}\equiv\frac{b}{d}\Mod{\frac{n}{d}}.$
 Thus, we get $ac\equiv b\Mod{n}.$
 So $(y^{a})^{c}=y^{ac}=y^{b}.$
 Therefore, we get $y^{a}\simfin  y^{b}.$	
	
	Now for all $a,b\in\mathbb{Z},$ $xy^{a}$ has order $2,$ and 
	\begin{align*}
	(xy^{b})\cdot xy^{a}\cdot (xy^{b})^{-1}	&=(xy^{b})\cdot xy^{a}\cdot y^{-b}x\\
																					&=xy^{b}xy^{a-b}x\\
																					&=xy^{b}x^2y^{b-a}\\
																					&=xy^{b}y^{b-a}\\
																					&=xy^{2b-a}
  \end{align*}
	and 
	\begin{align*}
	y^{b}\cdot xy^{a}\cdot y^{-b}	&=xy^{-b}y^{a-b}\\
																&=xy^{a-2b}.
	\end{align*}
	So we get 
	\begin{align*}
		[xy^{a}]_{\simfin }	&=\{xy^{a+2b}\ |\ b\in\mathbb{Z}\}\cup\{xy^{-a+2b}\ |\ b\in\mathbb{Z}\}\\
													&=\{xy^{a+2b}\ |\ b\in\mathbb{Z}\}.
	\end{align*}
	Hence, $xy^{a}\simfin  xy^{b}$ if and only if $\exists c\in\mathbb{Z}$ such that $a+2c\equiv b\Mod{n}.$
	Thus, we get
	\[\mathcal{F}_{G}= \begin{cases} 
      \mathcal{F}_{\mathbb{Z}/n}+1 & \text{if } n \text{ is odd}\\
      \mathcal{F}_{\mathbb{Z}/n}+2 & \text{otherwise}\\ 
   \end{cases}
\].
\end{proof}

\begin{Rem}
	Let $D_{\infty}$ be the infinite dihedral group. Let $x,y$ be the elements generating $D_{\infty}$ with relations $x^{2}=(xy)^{2}=1.$
	Since $D_{\infty}$ is virtually nilpotent (it contains the subgroup $\langle y\rangle\cong\mathbb{Z}$ of index 2), it is polynomially full by Corollary 
	~\ref{virtuallynilpotent}. Straightforward calculation shows that $\{1,x,xy\}$ is a complete set of representatives for the equivalence classes in $D^{fin}_{\infty}/_{\simfin} .$
	Hence, $\mathcal{F}_{D_{\infty}}=3.$
\end{Rem}

In the following, we give a formula of $\mathcal{F}_{G}$ for $G=\mathfrak{S}_n.$ 
\begin{Prop}
 For all $n\in\mathbb{N},$ $\mathcal{F}_{\mathfrak{S}_n}$ is equal to the number of conjugacy classes in $\mathfrak{S}_n,$ where $\mathfrak{S}_n$ is the symmetric group on a finite set of $n$ symbols.
\end{Prop}

\begin{proof}
 For all $g\in\mathfrak{S}_n$ and for all $a\in\mathbb{N}$ with $\operatorname{gcd}(a,\operatorname{order}(g))=1,$ 
 the permutations $g$ and $g^a$ have the same cycle structures. So they are conjugates.
 Now, $\forall g,h\in\mathfrak{S}_n,$ we have $g\simfin  h$ if and only if $\exists a\in\mathbb{N}$ with $\operatorname{gcd}(a,\operatorname{order}(g))=1$ and $g^{a}\in C^{\mathfrak{S}_n}(h).$
 Therefore, we get $g\simfin  h$ if and only if $g\in C^{\mathfrak{S}_n}(h).$
\end{proof}


\begin{thebibliography}{99}

\bibitem{BaumConnes} P. Baum, A. Connes, N. Higson. \textit{Classifying space for proper actions and $K$-theory of group $C^*$-algebras.}
  $C^*$-algebras: 1943-1993 (San Antonio, TX, 1993), 240--291, Contemp. Math., 167, Amer. Math. Soc., Providence, RI, 1994.\\

\bibitem{Engel} A. Engel. \textit{Banach strong Novikov conjecture for polynomially contractible groups.} arXiv:1702.02269v2 [math.KT].\\

\bibitem{Fgvb} H. Figueroa, J. M. Gracia-Bond\'ia, and J. C. V\'arilly. \textit{Elements of Noncommutative Geometry.}
  Birkh\"auser Advanced Texts, Birkh\"auser (2000).\\

\bibitem{Gong} S. Gong. \textit{Finite Part of Operator $K$-Theory for Groups with Rapid Decay.}
  Journal of Noncommutative Geometry 9 (2015), 697--706.\\
	
\bibitem{Gromov} M. Gromov. \textit{Groups of polynomial growth and expanding maps.}
	Inst. Hautes \'Etudes Sci. Publ. Math. No. 53 (1981), 53--73.\\
 
\bibitem{Ranicki} J. Ranicki. \textit{Algebraic and geometric Surgery.}
  Oxford Mathematical Monograph (2002).\\

\bibitem{Roe} J. Roe. \textit{An index theorem on open manifolds. I.}
  J. Differential Geom. 27 (1988), No. 1, 87--113.\\

\bibitem{Sch} L. B. Schweitzer. \textit{A short proof that $M_n(A)$ is local if $A$ is local and Fr\'echet.}
  Int. J. Math. (1992), 581--589.\\

\bibitem{St} J. Stallings. \textit{Centerless groups-an algebraic formulation of Gottlieb's theorem.}
  Topology 4 (1965), 129--134.\\

\bibitem{Yu} S. Weinberger, G. Yu. \textit{Finite part of operator $K$-theory for groups finitely embeddable into Hilbert space and the degree of non-rigidity of manifolds.}
  Geometry and Topology 19 (2015), 2767--2799.\\
	
\bibitem{WeinbergerXieYu} S. Weinberger, Z. Xie, G. Yu. \textit{Additivity of higher rho invariants and nonrigidity of topological manifolds.} arXiv:1608.03661 [math.KT].\\	
	
\bibitem{Wolf} J.A. Wolf. \textit{Growth of finitely generated solvable groups and curvature of riemannian manifolds.}
  Journal of Differential Geometry, 2 (4) (1968), 421--446.\\

\bibitem{Yu2} G. Yu. \textit{The Novikov conjecture for algebraic $K$-theory of the group algebra over the ring of Schatten class operators.}
  Advances in Mathematics, Vol. 307 (2017), 727--753.\\

\end{thebibliography}
\end{document}